\documentclass[12pt]{article}
\usepackage{array}
\usepackage{amsthm,amsmath,amssymb,color}
\usepackage{fullpage}

\usepackage{color}
\usepackage[normalem]{ulem}

\newtheorem{thm}{Theorem}[section]
\newtheorem{theorem}[thm]{Theorem}

\begin{document}

\title{An odd $[1,b]$-factor in regular graphs from eigenvalues}
\author{Sungeun Kim\thanks{Incheon Academy of Science and Arts, Korea, Incheon, 22009, tjddms9282@gmail.com}\, 
	Suil O\thanks{Department of Applied Mathematics and Statistics, The State University of New York, Korea, Incheon, 21985, suil.o@sunykorea.ac.kr. Research supported by NRF-2017R1D1A1B03031758 and by NRF-2018K2A9A2A06020345}\,
	Jihwan Park\thanks{Incheon Academy of Science and Arts, Korea, Incheon, 22009, bjihwan37@gmail.com}\,
	and Hyo Ree\thanks{Incheon Academy of Science and Arts, Korea, Incheon, 22009, reehyo2234@naver.com}
}

\maketitle

\begin{abstract}
An odd $[1,b]$-factor of a graph $G$ is a spanning subgraph $H$ such that
for each vertex $v \in V(G)$, $d_H(v)$ is odd and $1\le d_H(v) \le b$.
Let $\lambda_3(G)$ be the third largest eigenvalue of the adjacency matrix of $G$.
For positive integers $r \ge 3$ and even $n$, 
Lu, Wu, and Yang~\cite{LWY} proved a lower bound for $\lambda_3(G)$
in an $n$-vertex $r$-regular graph $G$ to gurantee the existence of an odd $[1,b]$-factor in $G$.
In this paper, we improve the bound; it is sharp for every $r$.\\

\noindent
\textbf{Keywords:} Odd $[1,b]$-factor, eigenvalues\\

\noindent
\textbf{AMS subject classification 2010:} 05C50, 05C70
\end{abstract}

\section{Introduction}
In this paper we deal only with finite and undirected graphs without loops or multiple edges. 
The {\it adjacency matrix} $A(G)$ of $G$ is the $n$-by-$n$
matrix in which entry $a_{i,j}$ is 1 or 0 according to whether $v_i$ and $v_j$ are adjacent or not, where $V(G) = \{v_1,\ldots, v_n\}$. The {\it eigenvalues} of $G$ are the eigenvalues of its adjacency matrix $A(G)$.
Let $\lambda_1(G),\ldots, \lambda_n(G)$ be its eigenvalues in nonincreasing order. Note that the spectral radius of $G$, written $\rho(G)$ equals $\lambda_1(G)$.

The degree of a vertex $v$ in $V(G)$, written $d_G(v)$, is the number of vertices adjacent to $v$.
An {\it odd} (or {\it even) $[a,b]$-factor} of a graph $G$ is a spanning subgraph $H$ of $G$ such that
for each vertex $v \in V(G)$, $d_H(v)$ is odd (or even) and $a \le d_H(v) \le b$; an $[a,a]$-factor is called the {\it $a$-factor}.
For a positive integer $r$, a graph is {\it $r$-regular} if every vertex has the same degree $r$. Note that $\lambda_1(G)=r$ if $G$ is $r$-regular. Many researchers proved the conditions for a graph to have an $a$-factor, or (even or odd) $[a,b]$-factor.
(See \cite{AKNT, K, M, N})
Brouwer and Haemers started to investiage the relations between eigenvalues and the existence of $1$-factor.


In fact, they~\cite{BH2} proved that 
if $G$ is an $r$-regular graph without an 1-factor, then 
$$\lambda_3(G) > \begin{cases}
r-1 + \frac 3{r+1} ~~\text{ if } r \text{ is even, }\\
r-1 + \frac 3{r+2} ~~\text{ if } r \text{ is odd}
\end{cases}$$
by using Tutte’s 1-Factor Theorem~\cite{T}, which is a special case of Berge-Tutte Formula~\cite{B}.
Cioab{\v a}, Gregory, and Haemers~\cite{CGH} improved their bound and in fact proved
that if $G$ is an $r$-regular graph without an 1-factor, then 
$$\lambda_3(G) \ge \begin{cases}
\theta = 2.85577... & \text{if } r=3, \\
\frac 12 (r-2 + \sqrt{r^2+12}) & \text{if } r\ge 4 \text{ is } \text{ even}, \\
\frac 12 (r-3 + \sqrt{(r+1)^2+16}) & \text{if } r\ge 5 \text{ is } \text{ odd},
\end{cases}$$
where $\theta$ is the largest root of $x^3 - x^2 - 6x +2=0$. 
More generally, O and Cioab{\v a}~\cite{CO} determined connections between the eigenvalues of a $t$-edge connected $r$-regular graph and its matching number when $1 \le t \le r - 2$. 
In 2010, Lu, Wu, and Yang~\cite{LWY} proved that if an $r$-regular graph $G$ with even number of vertices has no odd $[1,b]$-factor, then  $$\lambda_3(G) > \begin{cases}
r - \frac{\lceil \frac rb \rceil -2}{r+1}+\frac 1{(r+1)(r+2)}~~\text{ if } r \text{ is even and } \lceil \frac rb \rceil \text{ is even,}\\
r - \frac{\lceil \frac rb \rceil -1}{r+1}+\frac 1{(r+1)(r+2)}~~\text{ if } r \text{ is even and } \lceil \frac rb \rceil \text{ is odd,}\\
r - \frac{\lceil \frac rb \rceil -1}{r+1}+\frac 1{(r+2)^2}~~~~~~\text{ if } r \text{ is odd and } \lceil \frac rb \rceil \text{ is even,}\\
 r - \frac{\lceil \frac rb \rceil -2}{r+1}+\frac 1{(r+2)^2}~~~~~~\text{ if } r \text{ is odd and } \lceil \frac rb \rceil \text{ is odd.}
\end{cases}$$

To prove the above bounds in the paper~\cite{LWY}, they used Amahashi's result.

\begin{thm}{\rm \cite{A}}\label{A}
	Let $G$ be a graph and let $b$ be a positive odd integer. Then $G$ contains an odd $[1,b]$-factor
	if and only if for every subset $S\subseteq V(G)$, $o(G-S)\le b|S|$, where $o(H)$ is the number of odd components in a graph $H$.
\end{thm}

Thoerem~\ref{A} guarantees that if there is no odd $[1,b]$-factor in an $r$-regular graph, then there exists a subset $S \in V(G)$ such that $o(G-S) > b|S|$. By counting the number of edges between $S$ and $G-S$,
we can show that $G-S$ has at least three odd components $Q_1, Q_2, Q_3$ such that $|[V(Q_i),S]| \le r-1$
(see the proof of Theorem~\cite{LWY} or Theorem~\ref{main}). Then they found lower bounds for the largest eigenvalue in a graph in the family ${\cal F}_{r,b}$, where ${\cal F}_{r,b}$ is a family of such a possible component depending on $r$ and $b$, and those bounds are appeared above.

In this paper, we improve their bound and in fact prove that
if  $G$ is an $n$-vertex $r$-regular graph without an odd $[1,b]$-factor, then  $$\lambda_3(G) \ge \rho(r,b),$$
where
$$\rho(r,b)=\begin{cases}
	\frac{r-2+\sqrt{(r+2)^2-4(\lceil \frac rb\rceil-2)}}{2} & \text{ if both } r \text{ and } \lceil \frac rb\rceil \text{ are even, }\\
	
	\frac{r-2+\sqrt{(r+2)^2-4(\lceil \frac rb\rceil-1)}}{2} & \text{ if } r \text{ is even and } \lceil \frac rb \rceil \text{ is  odd, }\\
	
	\frac{r-3+\sqrt{(r+3)^2-4(\lceil \frac rb\rceil-2)}}{2} & \text{ if both } r \text{ and }  \lceil \frac rb\rceil \text{ are odd,} \\
	
	\frac{r-3+\sqrt{(r+3)^2-4(\lceil \frac rb\rceil-1)}}{2} & \text{ if } r \text{ is odd and } \lceil \frac rb\rceil \text{ is even.}
	\end{cases}$$

The bounds that we found are sharp in a sense that there exists a graph $H$ in ${\cal F}_{r,b}$ such that
$\lambda_1(H)=\rho(r,b)$.	




For undefined terms, see West~\cite{W} or Godsil and Royle~\cite{GR}.

\section{Construction}

Suppose that $\varepsilon=\begin{cases} 2 & \text{if } r \text{ and } \lceil \frac{r}{b} \rceil \text{ has same parity} \\ 1 & \text{otherwise} \end{cases}$ and $\eta= \lceil \frac{r}{b} \rceil - \varepsilon$.
In this section, we provide graphs $H_{r,\eta}$ such that $\lambda_1(H_{r,\eta})=\rho(r,b)$. These graphs show that the bounds in Theorem~\ref{main} are sharp. 

Now, we define the graph $H_{r,\eta}$ as follows: $$H_{r,\eta}=\begin{cases}
\mathrm{K}_{r+1-\eta} \vee \overline{\frac{\eta}{2}\mathrm{K}_2} &\text{if } r \text{ is even, }\\
\overline{\mathrm{C}_\eta} \vee \overline{\frac{r+2-\eta}{2}\mathrm{K}_2} & \text{if } r \text{ is odd. }
\end{cases}$$

To compute the spectral radius of $H_{r,\eta}$, 
the notion of “equitable partition” of a vertex set in a graph is used. Consider a partition $V(G) = V_1 \cup \cdots \cup V_s$ of the vertex set of a graph $G$ into $s$ non-empty subsets. For $1 \le i, j \le s$, let $q_{i,j}$ denote the average number of neighbours in $V_j$ of the vertices in $V_i$. The quotient matrix of this partition is the $s \times s$ matrix whose $(i, j)$-th entry equals $q_{i,j}$. The eigenvalues of the quotient matrix interlace the
eigenvalues of $G$. This partition is {\it equitable} if for each $1 \le i, j \le s$, any vertex $v \in V_i$ has exactly $q_{i,j}$ neighbours in $V_j$. In this case, the eigenvalues of the quotient matrix are eigenvalues of $G$ and the spectral radius of the quotient matrix equals the spectral radius of $G$ (see \cite{BH},\cite{GR} for more details).


\begin{theorem}
For $r \ge 3$ and $b \ge 1$, we have $\lambda_{1}(H_{r,\eta})=\rho(r,b)$.
\end{theorem}
\begin{proof} We prove this theorem only in the case when $r$ is odd because the proof of the other case is similar.

Consider the vertex partition $\{V(\overline{\mathrm{C}_\eta}),V(\overline{\frac{r+2-\eta}{2}\mathrm{K}_2})\}$ of $H_{r,\eta}$. The quotient matrix of the vertex partitions equals
	$$Q=\begin{pmatrix}
	\eta-3 & r+2-\eta \\
	\eta & r-\eta 
	\end{pmatrix}
	$$
	
The characteristic polynomail of $Q$ is
	$$ p(x)=(x-\eta+3)(x-r+\eta)-(r+2-\eta)\eta.$$
	
Since the vertex partition is equitable,  the largest root of the graph $H_{r,\eta}$ equals the largest root of the polynomial, which is 
$\lambda_1(Q) = \frac{ r-3 + \sqrt{ (r+3)^2 -4\eta } }2$. 
\end{proof}

\section{Main results}


In this section, we prove an upper bound for $\lambda_3(G)$ in an $r$-regular graph $G$ with even number of vertices to guarantee the existence of an odd $[1,b]$-factor by using Theorem~\ref{A} and Theorem~\ref{ind}.


\begin{thm}{\rm \cite{BH,GR}}\label{ind}
	If $H$ is an induced subgraph of a graph $G$, then $\lambda_i(H) \le \lambda_i(G)$ for all $i \in \{1,\ldots,|V(H)|\}$.
\end{thm}

\begin{thm}\label{main}
	Let $r \ge 3$, and $b$ be a positive odd integer less than $r$. If $\lambda_3(G)$ of an $r$-regular graph $G$ with even number of vertices is smaller than $\rho(r,b)$, then $G$ has an odd $[1,b]$-factor.
\end{thm}
\begin{proof}
We prove the contrapositive. Assume that an $r$-regular graph $G$ with even number of vertices has no odd $[1,b]$-factor.
By Theorem~\ref{A}, there exists a vertex subset $S\subseteq V(G)$ such that
$o(G-S) > b|S|$. Note that since $|V(G)|$ is even, $b$ is odd, and $o(G-S)\equiv |S|~(\!\!\!\mod 2)$, we have $o(G-S)\ge b|S|+2$. Let $G_1,\ldots,G_q$ be the odd components of $G-S$, where $q=o(G-S)$.\\

\noindent
{\it Claim 1. There are at least three odd components, say $G_1,G_2,G_3$, such that $|[V(G_i),S]|< \lceil\frac rb\rceil$ for all $i\in\{1,2,3\}$.}

 Assume to the contrary that there are at most two such odd components in $G-S$. Since $G$ is $r$-regular, we have
$$r|S| \ge \sum_{i=1}^q |[V(G_i),S]| \ge \lceil\frac rb\rceil(q-2)+2 \ge \lceil\frac rb\rceil b|S|+2 \ge r|S|+2,$$
which is a contradiction.

By Theorem~\ref{ind}, we have 
\begin{equation}
\lambda_3(G) \ge \lambda_3(G_1\cup G_2 \cup G_3) \ge \min_{i \in \{1,2,3\}}\lambda_1(G_i).
\end{equation}

Now, we prove that if $H$ is an odd component of $G-S$ such that $|[V(H),S]| < \lceil\frac rb\rceil$, then $\lambda_1(H) \ge \rho(r,b)$.\\

\noindent
{\it Claim 2. If $H$ is an odd components of $G-S$ such that $|[V(H),S]| < \lceil\frac rb\rceil$ and if $\lambda_1(H)\le \lambda_1(H')$ for all odd components $H'$ in $G-S$ such that $|[V(H'),S]| < \lceil\frac rb\rceil$, then we have $|V(H)|=\begin{cases}
r+2 \text{ if } r \text{ is odd,}\\
r+1 \text{ if } r \text{ is even}
	\end{cases}$, and ~$2|E(H)|=\begin{cases}
		r(r+2)-\eta \text{ if } r \text{ is odd,} \\
		r(r+1)-\eta \text{ if } r \text{ is even}.
\end{cases}$} 

Let $x=\begin{cases}
1 \text{ if } r \text{ is odd,}\\
0 \text{ if } r \text{ is even}.
\end{cases}$ 
Since $|[V(H),S]| < \lceil\frac rb\rceil < r$ and $G$ is $r$-regular, we have $|V(H)|\ge r+1+x$ since $H$ has an odd number of vertices.
If $|V(H)| > r+1+x$, then we have $|V(H)| \ge r+3+x$ since $H$ has an odd number of vertices. 
Thus it suffices to show $\rho(r,b) < \lambda_1(H)$ if $|V(H)| \ge r+3+x$. By using the fact that $\lambda_1(G) \ge \frac{2|E(G)|}{|V(G)|}$ for any graph $G$, we have


$$\lambda_1(H) > \frac{r|V(H)|-\eta}{|V(H)|} \ge \frac{r(r+3+x)-\eta}{r+3+x} > \frac{r-2-x+\sqrt{(r+2+x)^2-4\eta}}{2}.$$\\




Now, we prove this theorem by considering two cases depending on the parity of $r$.

{\it Case 1. $r$ is even.} By Claim 2, assume that $H$ is an odd component of $G-S$ such that $|[V(H),S]| < \lceil\frac rb\rceil$, $|V(H)|=r+1$, and $2|E(H)|=r(r+1)-\eta$. Then there are at least $r+1-\eta$ vertices of degree $r$. Let $V_1$ be a set of vertices with degree $r$ such that $|V_1|=r+1-\eta$, and let $V_2$ be the remaining vertices in $V(H)$. Then the quotient matrix of the vertex partition $\{V_1,V_2\}$ of $H$ equals
$$\begin{pmatrix}
r-\eta & \eta\\
r+1-\eta & \eta-2
\end{pmatrix}$$\\
whose characteristic polynomial is $p(x)=(x-r+\eta)(x-\eta+2)-\eta(r+1-\eta)$.
Since the largest root of $p(x)$ equals $\rho(r,b)$, we have
$\lambda_1(H) \ge \rho(r,b).$\\

{\it Case 2. $r$ is odd.} By Claim 2, assume that $H$ is an odd component of $G-S$ such that $|[V(H),S]| < \lceil\frac rb\rceil$, $|V(H)|=r+2$, and $2|E(H)|=r(r+2)-\eta$. Then there are at least $r+2-\eta$ vertices of degree $r$. Let $V_1$ be a set of vertices with degree $r$ such that $|V_1|=r+2-\eta$, and let $V_2$ be the remaining vertices in $V(H)$. Suppose that there are $m_{12}$ edges between $V_1$ and $V_2$. Note that $(r+2-\eta)(\eta-1) \le m_{12} \le (r+2-\eta)\eta$.
Then the quotient matrix of the vertex partion $\{V_1,V_2\}$ of $H$ equals
	$$
	\begin{pmatrix}
	r- \frac{m_{12}}{r+2-\eta} & \frac{m_{12}}{r+2-\eta} \\
	\frac{m_{12}}{\eta} & r-1-\frac{m_{12}}{\eta}
	\end{pmatrix}
	$$
	whose characteristic polynomial is $q(x)=(x-r+ \frac{m_{12}}{r+2-\eta})(x-r+1+\frac{m_{12}}{\eta})-\frac{m_{12}^2}{(r+2-\eta)\eta}$.
	
	Note that since $(r+2-\eta)(\eta-1) \le m_{12} \le (r+2-\eta)\eta$,  $m_{12}$ can be expressed $m_{12}=(r+2-\eta)\eta -t$, where $0 \le t \le r+2-\eta$. Thus we have
		$$q(x)=x^2-(r-3+\frac{t(r+2)}{(r+2-\eta)\eta})x-3r+\eta-\frac{t}{r+2-\eta}+\frac{tr(r+2)}{(r+2-\eta)\eta}$$$$=x^2-(r-3)x-3r+\eta-\frac{t(r+2)}{(r+2-\eta)\eta}x-\frac{t}{r+2-\eta}+\frac{tr(r+2)}{(r+2-\eta)\eta}.$$
	Note that $q(\rho(r,b)) =-\frac{t(r+2)}{(r+2-\eta)\eta} (\rho(r,b)+\frac{\eta}{r+2}-r) \le 0$, since $\eta \ge 1$ and $0 \le t \le r+2-\eta$.
	
	
	
	
	
	
\end{proof}


\end{document}